\documentclass[11pt]{amsart}

\usepackage{amsmath}
\usepackage{amssymb}
\usepackage{graphicx}
\usepackage[
citestyle=numeric,
bibstyle=numeric
]{biblatex}
\usepackage{biblatex}

\newtheorem{thm}{Theorem}[section]
\newtheorem{prop}[thm]{Proposition}
\newtheorem{lem}[thm]{Lemma}
\theoremstyle{definition}
\newtheorem{defn}[thm]{Definition}
\newtheorem{cor}{Corollary}[section]
\newtheorem{ex}[thm]{Example}

\numberwithin{equation}{section}

\title[Frame-like Fourier expansions for $\mathbb{R}$ ]{Frame-like Fourier expansions for finite Borel measures on $\mathbb{R}$}

\author[C. Berner]{Chad Berner}

\address[C. Berner]{Iowa State Unveristy, USA}
\email{{\tt cberner@iastate.edu}}



\begin{document}

\begin{abstract}
 It is known that if a finite Borel measure $\mu$ on $[0,1)$ possesses a frame of exponential functions for $L^{2}(\mu)$, then $\mu$ is of pure type. In this paper, we prove the existence of a class of finite Borel measures $\mu$ on $[0,1)$ that are not of pure type that possess frame-like Fourier expansions for $L^{2}(\mu)$. We also show properties and classifications of certain measures possessing this type of Fourier expansion. Additionally, we establish a frame-like Fourier expansion for $L^{2}(\mu)$ where $\mu$ is a singular Borel probability measure on $\mathbb{R}$. Finally, we show measures $\mu$ on $[0,1)$ that possess these frame-like Fourier expansions for $L^{2}(\mu)$ have all $f\in L^{2}(\mu)$ as $L^{2}(\mu)$ limits of harmonic functions with frame-like coefficients. We also discuss when the inner products of these expansions coincide with model spaces and subspaces of harmonic functions on the disk.
\end{abstract}

\maketitle

\section{Introduction}

We are interested in describing and proving the existence of a class of finite Borel measures $\mu$ whose $L^{2}$ spaces can have their elements expressed as Fourier expansions whose coefficients come from inner products. Substantial work has been provided for finding measures $\mu$ where a sequence of exponential functions forms an orthonormal basis of $L^{2}(\mu)$, and it is clear that Lebesgue measure on $[0,1)$ is an example of this. Additionally, there are even examples of singular measures $\mu$ which possess a sequence of exponential functions that form an orthonormal basis of $L^{2}(\mu)$ \cite{Jorgensen1998Dense}.

Furthermore, there is recent work in understanding which measures $\mu$ possess a sequence of exponential functions that form a frame for $L^{2}(\mu)$.
If a space $L^{2}(\mu)$ admits a frame consisting of exponential functions, there is a canonical dual frame sequence in $L^{2}(\mu)$ so that every element of $L^{2}(\mu)$ can be expressed as a Fourier expansion. Initially, Duffin and Schaffer provided substantial work in Fourier series from introducing frame theory in the Fourier analysis \cite{Duffin1952Class}. More recently however, in the case that $\mu$ is absolutely continuous, we know exactly when it possesses a frame of exponential functions \cite{Lai2011Fourier}. It is also known that if $\mu$ possesses a frame of exponential functions, then $\mu$ is of pure type, that is singular or absolutely continuous \cite{He2013Exponential}. While there are examples known of measures $\mu$ that are singular which possess a frame of exponential functions such as in \cite{Picioroaga2017Fourier}, when exactly measures $\mu$ have this property is still an open problem.
For more discussion on singular measures $\mu$ that possess a sequence of exponential functions that form a frame of $L^{2}(\mu)$, one can also read the work of Dutkay, Han, Sun, and Weber \cite{Dutkay2011Beurling}.

The purpose of this paper is to prove the existence of a class of finite Borel measures $\mu$ on $[0,1)$ that are not of pure type which possess a sequence $\{g_{n}\}_{n=-\infty}^{\infty}\subseteq L^{2}(\mu)$
such that for all $f\in L^{2}(\mu)$,
\begin{equation}\label{prp}
   f=\sum_{n=-\infty}^{\infty}\langle f,g_{n}\rangle e^{2\pi i nx}. 
\end{equation}
We also want to understand characteristics of measures with this property.
Furthermore, we are specifically interested in the case where the frequencies of the exponential functions come from $\mathbb{Z}$.
The secondary purpose of this paper is to provide another type of frame-like Fourier expansion for $L^{2}(\mu)$ where $\mu$ is a singular Borel probability measure on $\mathbb{R}$.

The outline of this paper is as follows: section two is some known results and early observations of the frame-like expansions we are interested in from equation (\ref{prp}). 
Section three is a few results on classifying certain measures with the property from equation (\ref{prp}) and after this, we prove the existence of a class of measures with this property that are not of pure type. Additionally, in section four we prove the existence of another type of frame-like Fourier expansion for elements of $L^{2}(\mu)$ where $\mu$ is a singular Borel probability measure on $\mathbb{R}$. Finally, section five of this paper is establishing some connections between Bessel sequences and subspaces of harmonic functions on the disk. In particular, we show that if a measure has the property from equation (\ref{prp}) where $\{g_{n}\}$ is Bessel, then every $f\in L^{2}(\mu)$ is the $L^{2}(\mu)$ boundary of a harmonic function whose coefficients come from $\{g_{n}\}$.

\section{Preliminaries}
We begin by recalling some frame terminology in a Hilbert space.
\begin{defn}
Let $\{g_{n}\}\subseteq H$ where $H$ is a Hilbert space.
\begin{enumerate}
    \item If there exists $A,B>0$ such that
    $$A||f||^{2}\leq \sum_{n}|\langle f, g_{n}\rangle|^{2}\leq B||f||^{2}$$
    for all $f\in H$, then $\{g_{n}\}$ is called a \textbf{frame}.
    \item If $\{g_{n}\}$ is a frame and $A=B=1$, then $\{g_{n}\}$ is called a \textbf{Parseval frame}.
    \item If $\{g_{n}\}$ is a frame and if
$$\sum c_{n}g_{n}=0 \implies \{c_{n}\}=0,$$ where $\{c_{n}\}\in \ell^{2}$, we say $\{g_{n}\}$ is a \textbf{Riesz basis}.
\item If $A$ or $B$ exist, we call $\{g_{n}\}$ a \textbf{semi-lower frame} or a \textbf{Bessel sequence}, respectively. 
\end{enumerate}
Also, for Bessel sequence $\{g_{n}\}$, the map $g\to \{\langle g,g_{n}\rangle\}$ is the \textbf{analysis operator} associated with $\{g_{n}\}$, and its adjoint map: $\{c_{n}\}\to \sum_{n}c_{n}g_{n}$ is called the \textbf{synthesis operator}.
Furthermore, if $\{g_{n}\}$ is a frame, the map $g\to \sum_{n}\langle g,g_{n}\rangle g_{n}$ is the \textbf{frame operator}, and it is positive definite and invertible.
\end{defn}
Now we define what it means for a sequence in a Hilbert space to have a \textbf{dextrodual}. This definition is similar to a definition of psudeo-duals discussed by Li and Ogawa in \cite{Li2001Pseudo-duals}, and it is the frame-like Fourier expansions we are discussing in section three.
\begin{defn}
    We say that sequence $\{g_{n}\}$ is \textbf{dextrodual} to sequence $\{x_{n}\}$ in Hilbert space $H$ if for all $x\in H$,
    $$x=\sum_{n}\langle x,g_{n}\rangle x_{n}$$ where the convergence is in norm and may be conditional.
\end{defn}

It is clear that if $\{x_{n}\}$ is a frame in Hilbert space $H$, then its canonical dual frame sequence is dextrodual to $\{x_{n}\}$. However as we will see soon, a sequence having a dextrodual sequence is strictly weaker than the sequence being a frame. 
\subsection{Dextrodual observations}
We explore some initial observations of this notion in a Hilbert space.
\begin{lem}\label{RB}
If a sequence $\{g_{n}\}$ is a Riesz basis in a Hilbert space $H$ and is dextrodual to sequence $\{x_{n}\}$, then $\{x_{n}\}$ is the canonical dual frame of $\{g_{n}\}$. In particular, $\{x_{n}\}$ is a Riesz basis.
\end{lem}

\begin{proof}
Notice that the analysis operator associated with $\{g_{n}\}$ is surjective since the range of this operator is closed from $\{g_{n}\}$ being a frame, and it has dense range from its adjoint being injective.
Therefore, for any $\{c_{n}\}\in \ell^{2}(\mathbb{N})$,
$\sum_{n}c_{n}x_{n}$ converges so that the synthesis operator for $\{x_{n}\}$ is bounded from being a point-wise limit of bounded operators.
Then by an adjoint calculation, we get for all $f\in H$,
$$f=\sum_{n}\langle f,x_{n}\rangle g_{n}.$$

Let $\{h_{n}\}$ be the canonical dual frame of $\{g_{n}\}$, we have for all $f\in H$,
$$\sum_{n}\langle f,x_{n}-h_{n}\rangle g_{n}=0\implies \{x_{n}\}=\{h_{n}\}$$ since $\{g_{n}\}$ is a Riesz basis.

Finally, note that the synthesis operator for $\{x_{n}\}$ is the left inverse of the analysis operator for $\{g_{n}\}$. Then the synthesis operator for $\{x_{n}\}$ is injective since the analysis operator for $\{g_{n}\}$ is invertible, showing $\{x_{n}\}$ is a Riesz basis.
\end{proof}
The following Lemma is an easy observation discussed in \cite{Li2001Pseudo-duals}.
\begin{lem}\label{LFB}
If $\{g_{n}\}$ is a Bessel sequence in Hilbert space $H$ with bound $B$ that is dextrodual to $\{x_{n}\}$, then $\{x_{n}\}$ is a semi-lower frame with bound at least $\frac{1}{B}$.
\end{lem}

\subsection{Known results}
Now we establish some known results that will prove useful in our discussion.
\begin{thm}[He, Xing-Gang and Lai, Chun-Kit and Lau, Ka-Sing]\label{notbessel}
If $\{e^{2\pi i \lambda x}\}_{\lambda \in \Lambda}$ is a frame in $L^{2}(\mu)$ for $\mu$ that is a singular Borel probability measure on $[0,1)$ that is not discrete, then $D^{-}(\Lambda)=0$ where $D^{-}$ denotes lower Beurling density. In particular, If $\mu$ is a singular Borel probability measure on $[0,1)$, then $\{e^{2\pi i nx}\}_{n=0}^{\infty}$ is not a Bessel sequence in $L^{2}(\mu)$.
\end{thm}
The following was proven in \cite{Herr2017Fourier} but was inspired by the work of Kwapien and Mycielski \cite{Kwapien2001Kaczmarz}.
\begin{thm}[Herr and Weber]\label{singular}
    If $\mu$ is a singular Borel probability measure on $[0,1)$, then there is a Parseval frame $\{g_{n}\}_{n=0}^{\infty}$ that is dextrodual to $\{e^{2\pi i nx}\}_{n=0}^{\infty}$.
\end{thm}

These two theorems show that having a dextrodual sequence is strictly weaker than being a frame. 
We also know exactly when $\{e^{2\pi i nx}\}_{n=-\infty}^{\infty}$ forms a frame for $L^{2}(\mu)$ when $\mu$ is an absolutely continuous Borel probability measure on $[0,1)$ by the following:
\begin{thm}[Lai, Chun-Kit]\label{absc}
    If $\mu$ is an absolutely continuous Borel probability measure on $[0,1)$, then $\{e^{2\pi i nx}\}_{n=-\infty}^{\infty}$ is a frame for $L^{2}(\mu)$ if and only if the Radon-Nykodym derivative of $\mu$ is bounded above and below on its support.
\end{thm}

\section{Measures on the torus with frame-like Fourier expansions }
Our goal is to completely characterize measures $\mu$ that have the property from equation (\ref{prp}). In particular, we prove the existence of a class of measures on the torus that are not of pure type that possess this property in this section.

\subsection{Properties of measures with frame-like Fourier expansions}
However, we did not achieve a complete characterization of these measures. Instead,
we begin this section with presenting some properties of finite Borel measures $\mu$ on $[0,1)$ where $\{e^{2\pi i nx}\}_{n=-\infty}^{\infty}$ has a dextrodual sequence $\{g_{n}\}$ in $L^{2}(\mu)$.
The first thing one can notice is that if a measure has this property, then its singular and absolutely continuous part have it as well.

\begin{lem}\label{bothparts}
    If $\mu=\mu_{a}+\mu_{s}$ is a finite Borel measure on $[0,1)$ where $\mu_{a}$ is absolutely continuous with Radon-Nikodym  derivative $g$ and $\mu_{s}$ is singular such that there is sequence $\{g_{n}\}$ that is dextrodual to $\{e^{2\pi i nx}\}_{n=-\infty}^{\infty}$ in $L^{2}(\mu)$, then $\{g_{n}\}$ is dextrodual to $\{e^{2\pi i nx}\}_{n=-\infty}^{\infty}$ in $L^{2}(\mu_{s})$ and $\{g_{n}\}$ is dextrodual to $\{e^{2\pi i nx}\}_{n=-\infty}^{\infty}$ in $L^{2}(\mu_{a})$.

Furthermore, if $\{g_{n}\}$ is a semi-lower frame or Bessel sequence for $L^{2}(\mu)$, then it is also a semi-lower frame or Bessel sequence for $L^{2}(\mu_{s})$ and $L^{2}(\mu_{a})$.
\end{lem}
\begin{proof}
Let $f\in L^{2}(\mu_{s})$ and $S$ be a Borel set of Lebesgue measure zero where $\mu_{s}(S^{c})=0$.
Define $\tilde{f}\in L^{2}(\mu)$ to be equal to $f$ in $S$ $\mu_{s}$ almost everywhere and equal to $0$ in $S^{c}$ almost everywhere with respect to Lebesgue measure.
Then
$$\tilde{f}=\sum_{n}\langle \tilde{f},g_{n}\rangle_{\mu}e^{2\pi i nx}=\sum_{n}\langle f,g_{n}\rangle_{\mu_{s}}e^{2\pi i nx}$$
where the convergence in is $L^{2}(\mu)$, so it is also in $L^{2}(\mu_{s})$.

Showing $\{g_{n}\}$ has semi-lower or Bessel bounds in $L^{2}(\mu_{s})$ is simple since if $\{g_{n}\}$ has semi-lower or Bessel bounds in $L^{2}(\mu)$,
$$||f||_{L^{2}(\mu_{s})}^{2}=||\tilde{f}||_{L^{2}(\mu)}^{2}\approx ||\{\langle \tilde{f},g_{n}\rangle_{\mu}\}||_{\ell^{2}}^{2}=||\{\langle f,g_{n}\rangle_{\mu_{s}}\}||_{\ell^{2}}^{2}$$

    The absolutely continuous case is similar.
\end{proof}
Now we will see that if the absolutely continuous part of $\mu$ is equivalent to Lebesgue measure, then there cannot be a positive singular part of $\mu$ and have $\mu$ with this dextrodual property from equation (\ref{prp}).
\begin{prop}\label{Lebesguepart}
    Let $\mu=\mu_{a}+\mu_{s}$ where $\mu_{s}$ is a non-zero finite singular Borel measure on $[0,1)$ and $\mu_{a}$ is an absolutely continuous Borel measure on $[0,1)$ with Radon-Nikodym derivative that is bounded above and below. For $S$ that is a Borel set of Lebesgue measure zero where $\mu_{s}(S^{c})=0$, there is no $\{c_{n}\}_{n=-\infty}^{\infty}$ such that
    $$\sum_{n}c_{n}e^{2\pi i nx}=\chi_{S}$$ with convergence in $L^{2}(\mu)$.
\end{prop}
\begin{proof}
    If $\chi_{S}$ could be expressed as a Fourier series in this manner, the convergence must also be in $L^{2}([0,1))$ since convergence in $L^{2}(\mu_{a})$ is equivalent to convergence in $L^{2}([0,1))$. Then $\{c_{n}\}=0$, which is a contraction.
\end{proof}
We can now classify measures with Radon-Nikodym derivative bounded below that have the property from equation (\ref{prp}) in the following way:
\begin{thm}
    Let $\mu$ be a finite Borel measure on $[0,1)$ with Radon-Nikodym derivative $g$ that is bounded below. Suppose that $\{g_{n}\}$ is a semi-lower frame that is dextrodual to $\{e^{2\pi i nx}\}_{n=-\infty}^{\infty}$ in $L^{2}(\mu)$. Then $g$ is bounded, and therefore, $\mu$ is absolutely continuous. 
\end{thm}
\begin{proof}
By Lemma \ref{bothparts}, $\{g_{n}\}$ is a semi-lower frame and dextrodual to $\{e^{2\pi i nx}\}_{n=-\infty}^{\infty}$ in $L^{2}(\mu_{a})$ where $\mu_{a}$ is the absolutely continuous part of $\mu$.
    Now for all $f\in L^{2}(\mu_{a}) $,
    $$\frac{f}{\sqrt{g}}=\sum_{n}\langle f, g_{n}\rangle \frac{1}{\sqrt{g}}e^{2\pi i nx},$$
    which follows since $\frac{1}{\sqrt{g}}$ is bounded.
    Also, it is easy to see that $\{\frac{e^{2\pi i nx}}{\sqrt{g}}\}_{n=-\infty}^{\infty}$ is an orthonormal basis of $L^{2}(\mu_{a})$. Then we must have by uniqueness of coefficients from an orthonormal basis,
    $$\langle f, g_{n}\rangle =\langle \frac{f}{\sqrt{g}},\frac{e^{2\pi i nx}}{\sqrt{g}}\rangle$$ so that
    $$g_{n}=\frac{e^{2\pi i nx}}{g}$$ for all $n$.
    
    Now if we assume that $g$ is unbounded, we show $\{g_{n}\}$ can't be a semi-lower frame for $L^{2}(\mu_{a})$. For each $n>0$, find Borel set $E_{n}$ of positive Lebesgue measure where $g>n$ on $E_{n}$.
    We have $$||\chi_{E_{n}}||_{\mu_{a}}^{2}\geq n|E_{n}|$$ while
    $$\sum_{k}|\langle \chi_{E_{n}}, g_{k}\rangle_{\mu_{a}} |^{2}=|E_{n}|.$$ There can't be $A>0$ such that
    $$An\leq 1$$ for all $n$, so $\{g_{n}\}$ can't be a semi lower frame, and $g$ is therefore bounded. By Proposition \ref{Lebesguepart}, $\mu$ must be absolutely continuous.
\end{proof}
We can also classify measures who have dextroduals to the exponential functions that are Riesz bases:

\begin{thm}
Let $\mu$ be a finite Borel measure on $[0,1)$. Then $\{e^{2\pi inx}\}_{n=-\infty}^{\infty}$ has a dextrodual sequence that is a Riesz basis if and only if $\mu$ is absolutely continuous with Radon-Nikodym derivative bounded above and below.
\end{thm}
\begin{proof}
If $\mu$ is absolutely continuous with Radon-Nikodym derivative bounded above and below, it is easy to show $\{e^{2\pi inx}\}_{n=-\infty}^{\infty}$ is a Riesz basis for $L^{2}(\mu)$, and its canonical dual frame is a Riesz basis as well.

If $\{e^{2\pi inx}\}_{n=-\infty}^{\infty}$ has a dextrodual sequence that is a Riesz basis, then
by Lemma \ref{RB}, $\{e^{2\pi inx}\}_{n=-\infty}^{\infty}$ is a Riesz basis in $L^{2}(\mu)$ and by Lemma \ref{bothparts}, $\{e^{2\pi inx}\}_{n=-\infty}^{\infty}$ is a frame for $L^{2}(\mu_{s})$ where $\mu_{s}$ is the singular part of $\mu$, which implies $\mu$ is absolutely continuous by Theorem \ref{notbessel}. Furthermore, the Radon-Nikodym  derivative of $\mu$ is bounded by Theorem \ref{absc}.

Suppose that the Radon-Nikodym derivative $g$ of $\mu$ is zero on a set of Lebesgue positive measures say, $S$. Then
$$\chi_{S}=\sum \langle \chi_{S},e^{2\pi inx}\rangle_{L^{2}[0,1)}e^{2\pi inx}$$ and this convergence is in $L^{2}(\mu)$ as well since $g$ is bounded, which contradicts uniqueness of coefficients for a Riesz basis. Then $g>0$ on $[0,1)$, and by Theorem \ref{absc}, we are done.
\end{proof}
Now we close this subsection with a more general description of measures with the property from equation (\ref{prp}).
\begin{thm}
Let $\mu$ be a finite Borel measure on $[0,1)$ with Radon-Nikodym  derivative $g$. Suppose that there is a Bessel sequence $\{g_{n}\}$ with bound $B$ that is dextrodual to $\{e^{2\pi i nx}\}_{n=-\infty}^{\infty}$ in $L^{2}(\mu)$, then $$g(x)>\frac{1}{3B}$$ almost everywhere on its support.
\end{thm}

\begin{proof}
By Lemma \ref{bothparts}, $\{g_{n}\}$ is a Bessel sequence that is dextrodual to $\{e^{2\pi inx}\}_{n=-\infty}^{\infty}$ in $L^{2}(\mu_{a})$ where $\mu_{a}$ is the absolutely continuous part of $\mu$.
Then $\{e^{2\pi inx}\}_{n=-\infty}^{\infty}$ is a semi-lower frame in $L^{2}(\mu_{a})$ with bound at least $\frac{1}{B}$ by Lemma \ref{LFB}.

Suppose that $g(x)\leq \frac{1}{3B}$ on a set of positive Lebesgue measure $S$ in its support. We have
$$S=\bigcup_{n=1}^{\infty}E_{n}$$ where $E_{n}=\{x\in S : \frac{\frac{1}{3B}}{n+1}<g(x)\leq \frac{\frac{1}{3B}}{n}\}$. There exists a $k\geq 1$ such that $E_{k}$ has positive Lebesgue measure. We have
$$\frac{\frac{1}{3B}}{(k+1)B}|E_{k}|\leq \frac{1}{B}\int_{E_{k}}gdx\leq \sum_{n}|\langle \chi_{E_{k}},e^{2\pi i nx}\rangle_{\mu_{a}}|^{2}$$
$$=\int_{E_{k}}g^{2}dx\leq \frac{(\frac{1}{3B})^{2}}{k^{2}}|E_{k}|\implies
\frac{1}{3B}\geq \frac{k^{2}}{(k+1)B},$$ which is a contradiction.
\end{proof}

\subsection{Classes of measures with frame-like Fourier expansions}

Now we show there are finite Borel measures $\mu$ on $[0,1)$ that have Bessel sequences $\{g_{n}\}_{n=-\infty}^{\infty}$ that are dextrodual to $\{e^{2\pi i nx}\}_{n=-\infty}^{\infty}$ in $L^{2}(\mu)$, even where $\mu$ is not of pure type.

By Theorem \ref{absc}  and Theorem \ref{singular}, we have the following two examples:

\begin{ex}[Lai, Chun-Kit]
 If $\mu$ is an absolutely continuous Borel measure on $[0,1)$ with Radon-Nikodym  derivative bounded above and below on its support, then there is a frame sequence that is dextrodual to $\{e^{2\pi i nx}\}_{n=-\infty}^{\infty}$ in $L^{2}(\mu)$.
\end{ex}

\begin{ex}[Herr and Weber]\label{sc}
    If $\mu$ is a finite singular Borel measure on $[0,1)$, then there is a frame sequence that is dextrodual to $\{e^{2\pi i nx}\}_{n=-\infty}^{\infty}$ in $L^{2}(\mu)$.
\end{ex}
Note Example \ref{sc} follows from scaling a Parseval frame from Theorem \ref{singular} and defining the sequence to be zero on negative indices.

Finally, we will prove the existence of a class of measures with the property from equation (\ref{prp}) that are not of pure type.
It turns out that if the absolutely continuous part of the measure is supported strictly inside interval $[0,1)$, there is still hope. In particular, we have an example that can easily be expanded to a class of measures:

\begin{prop}
    Let $\mu=\frac{1}{2}\delta_{\frac{1}{4}}+\mu_{a}$ where $\mu_{a}$ is a finite absolutely continuous Borel measure on $[0,1)$ with Radon-Nikodym  derivative $\chi_{[\frac{1}{2},1)}$ and $\delta_{\frac{1}{4}}$ is the point mass singular Borel probability measure at $\frac{1}{4}$.  There is a Parseval frame $\{g_{n}\}$ that is dextrodual to $\{e^{2\pi i nx}\}_{n=-\infty}^{\infty}$ in $L^{2}(\mu)$.
\end{prop}
\begin{proof}
    Let $f\in L^{2}(\mu)$ and define 
    $$\tilde{f}=f(\frac{1}{4})\chi_{[0,\frac{1}{2})}+f\chi_{[\frac{1}{2},1)}\in L^{2}([0,1)).$$
    Notice that 
    $$||\tilde{f}||_{L^{2}([0,1))}^{2}= \frac{1}{2}|f(\frac{1}{4})|^{2}+||f||_{\mu_{a}}^{2}=||f||_{d\mu}^{2}.$$

    Now see that $S_{M}(\tilde{f})=\sum_{n=-M}^{n=M}\langle \tilde{f},e^{2\pi i nx}\rangle_{L^{2}([0,1))} e^{2\pi i nx}$ converges to $f$ in $L^{2}(\mu)$ as $M\to \infty$.
    This is because $S_{M}(\tilde{f})(\frac{1}{4})\to f(\frac{1}{4})$ as $M\to \infty$ since $\tilde{f}$ is differentiable at $x=\frac{1}{4}$.

    Now since the map $f\to \tilde{f}$ is a linear isometry, for all $n$, there is $g_{n}\in L^{2}(\mu)$ where
    $$\langle f, g_{n}\rangle_{\mu} =\langle \tilde{f},e^{2\pi i nx}\rangle_{L^{2}([0,1))},$$ and $\{g_{n}\}$ is a Parseval frame in $L^{2}(\mu)$
    since $\{e^{2\pi i nx}\}_{n=-\infty}^{\infty}$ is a Parseval frame in $L^{2}([0,1))$.
\end{proof}
The general statement with a similar proof follows:
\begin{prop}\label{discretecase}
Let $\mu=\sum_{k=1}^{n}a_{k}\delta_{b_{k}}+\mu_{a}$ where $a_{k}>0$, $0<b_{k}<c$, $g$ is a Borel function bounded above and below almost everywhere, and $\mu_{a}$ is an absolutely continuous Borel measure on $[0,1)$ with Radon-Nikodym derivative $\chi_{[c,1)}g$.
There is a frame $\{g_{n}\}$ that is dextrodual to $\{e^{2\pi i nx}\}$.
\end{prop}
We can also extend this result further with the following statement:
\begin{thm}\label{examplecase}
    Let $\mu$ be a finite Borel measure on $[0,1)$ with Radon-Nikodym derivative $g$ supported in some closed interval $[c,d]$ where $c>0$, and let $g$ be bounded above and below on its support. Suppose also there is a closed interval $[a,b]$ disjoint from $[c,d]$ where $a>0$ such that
    $\mu_{s}([a,b]^{c})=0$ where $\mu_{s}$ is the singular part of $\mu$. Then there is $\{g_{n}\}$ Bessel sequence that is dextrodual to $\{e^{2\pi i nx}\}_{n=-\infty}^{\infty}$ in $L^{2}(\mu).$

    Furthermore, if $L^{2}(\mu_{s})$ is finite-dimensional, the constructed $\{g_{n}\}$ is a frame.
\end{thm}

\begin{proof}
 Let $f\in L^{2}(\mu)$. By Example \ref{sc}, there is a frame $\{h_{n}\}_{n=-\infty}^{\infty}$ such that
 $\sum_{n}\langle f, h_{n}\rangle_{\mu_{s}}e^{2\pi i nx}\to f$ in $L^{2}(\mu_{s})$.
 Now let $I_{1}$ and $I_{2}$ be two disjoint open intervals containing $[a,b]$ and $[c,d]$ respectively.
 Since $\{h_{n}\}$ is a Bessel sequence, there is $\tilde{f}\in L^{2}([0,1))$ 
 with $\langle f, h_{n}\rangle_{\mu_{s}}$ as its $n$th Fourier coefficient.
Then we have
$$||\sum_{n=-M}^{M}\langle \tilde f\chi_{I_{1}}, e^{2\pi i nx}\rangle_{L^{2}([0,1))} e^{2\pi i nx}-f||_{\mu_{s}}\leq$$
$$||S_{M}(\tilde f\chi_{I_{1}})-S_{M}(\tilde{f})||_{\mu_{s}}+||S_{M}(\tilde{f})-f||_{\mu_{s}}$$
so that $$||\sum_{n=-M}^{M}\langle \tilde f\chi_{I_{1}}, e^{2\pi i nx}\rangle_{L^{2}([0,1))} e^{2\pi i nx}-f||_{\mu_{s}}\to 0$$ as $M\to \infty$. This follows because $||S_{M}(\tilde{f})-f||_{\mu_{s}}\to 0$ by definition of $\{h_{n}\}$, and $||S_{M}(\tilde f\chi_{I_{1}})-S_{M}(\tilde{f})||_{\mu_{s}}\to 0$
because $S_{M}(\tilde{f}\chi_{I_{1}^{c}})$ converges to zero on $[a,b]$ uniformly. Also note that $S_{M}(\tilde{f}\chi_{I_{1}})$ converges to zero uniformly on $[c,d]$ so that $S_{M}(\tilde{f}\chi_{I_{1}})$ converges to zero in $L^{2}(\mu_{a})$ as well where $\mu_{a}$ is the absolutely continuous part of $\mu$.

 Similarly, $$\sum_{n=-M}^{M}\langle f, \chi_{\{x:g(x)>0\}}\frac{e^{2\pi i nx}}{g}\rangle_{\mu_{a}}e^{2\pi i nx}\to f$$ in $L^{2}(\mu_{a})$. Furthermore, 
 
 $\sum_{n=-M}^{M}\langle f, \chi_{\{x:g(x)>0\}}\frac{e^{2\pi i nx}}{g}\rangle_{\mu_{a}}e^{2\pi i nx}\to 0$ uniformly on $[a,b]$ so that the convergence is also in $L^{2}(\mu_{s}).$

 Then our Bessel sequence that is dextrodual to $\{e^{2\pi i nx}\}_{n=-\infty}^{\infty}$ in $L^{2}(\mu)$ is $$\{\chi_{I_{1}}A^{*}(\chi_{I_{1}}e^{2\pi i nx})+\chi_{\{x:g(x)>0\}}\frac{e^{2\pi i nx}}{g}\}$$ where $A:L^{2}(\mu_{s})\to L^{2}([0,1))$ is the bounded operator that maps $f\to \tilde{f}$. Note $A$ is bounded below since $\{h_{n}\}$ is a frame.

 Now notice that by Pythagorean theorem,
 $$\sum_{n}|\langle f, \chi_{I_{1}}A^{*}(\chi_{I_{1}}e^{2\pi i nx})+\chi_{\{x:g(x)>0\}}\frac{e^{2\pi i nx}}{g}\rangle_{\mu} |^{2}=$$
 $$\sum_{n}|\langle A(f)\chi_{I_{1}},e^{2\pi i nx}\rangle_{L^{2}([0,1))}+\langle f\chi_{\{x:g(x)>0\}},e^{2\pi i nx}\rangle_{L^{2}([0,1))}|^{2}=$$
 $$||A(f)\chi_{I_{1}}||_{L^{2}([0,1))}^{2}+||f\chi_{\{x:g(x)>0\}}||_{L^{2}([0,1))}^{2}.$$
 Therefore, $\{\chi_{I_{1}}A^{*}(\chi_{I_{1}}e^{2\pi i nx})+\chi_{\{x:g(x)>0\}}\frac{e^{2\pi i nx}}{g}\}$ is a frame if multiplication by $\chi_{I_{1}}$ on $Im(A)$ is bounded below. Also, because $S_{M}(A(f)\chi_{I_{1}})$ converges to $f$ in $L^{2}(\mu_{S})$, multiplication by $\chi_{I_{1}}$ on $Im(A)$ is injective. Then $\{\chi_{I_{1}}A^{*}(\chi_{I_{1}}e^{2\pi i nx})+\chi_{\{x:g(x)>0\}}\frac{e^{2\pi i nx}}{g}\}$ is a frame if $\chi_{I_{1}}Im(A)$ is a closed set. In particular, if $Im(A)$ is finite-dimensional, $\{\chi_{I_{1}}A^{*}(\chi_{I_{1}}e^{2\pi i nx})+\chi_{\{x:g(x)>0\}}\frac{e^{2\pi i nx}}{g}\}$ is a frame.
\end{proof}
Now conversely, the idea that the absolutely continuous part and singular part of the measure must be separated for this dextrodual property from equation (\ref{prp}) to hold is also illustrated with the following result:
\begin{thm}
    Let $\mu=\mu_{a}+\mu_{s}$ be a finite Borel measure on $[0,1)$ such that $\mu_{s}$ is singular and $\mu_{a}$ is absolutely continuous with Radon-Nykodym derivative $g$ bounded above and below on its support where the support of $g$ contains an interval $(c,d)$. Suppose there exists a Borel set $B$ of Lebesgue measure zero and $\mu$ positive measure such that
    $$B\subseteq [a,b]\subseteq (c,d)\subseteq \{x: g(x)>0\}$$ for some $a$ and $b$. Then there is no $\{c_{n}\}_{n=-\infty}^{\infty}\in \ell^{2}(\mathbb{N})$ such that
    $$\sum_{n}c_{n}e^{2\pi i nx}=\chi_{B}$$ with convergence in $L^{2}(\mu)$.
\end{thm}
\begin{proof}
Suppose there is $\{c_{n}\}_{n=-\infty}^{\infty}\in \ell^{2}(\mathbb{N})$ such that
    $$\sum_{n}c_{n}e^{2\pi i nx}=\chi_{B}$$ with convergence in $L^{2}(\mu)$. Let $$f=\sum_{n}c_{n}e^{2\pi i nx}$$ where the limit is taken in $L^{2}([0,1))$.

    We claim $f=0$ on $\{x: g(x)>0\}$. Otherwise, let $A\subseteq \{x: g(x)>0\}$ be a set of positive Lebesgue measure where $|f|>0$.
    We have $\sum_{n=-M}^{M}c_{n}e^{2\pi i nx}\to f$ point-wise almost everywhere in $[0,1)$ by Carleson's Theorem. Then by Egorov's theorem, there is a set of positive Lebesgue measure $C\subseteq A$ so that 
    $\sum_{n=-M}^{M}c_{n}e^{2\pi i nx}\to f$ uniformly on $C$. Then
    $$\lim_{M\to \infty}\int_{C}|\sum_{n=-M}^{M}c_{n}e^{2\pi i nx}|^{2}dx=\int_{C}|f|^{2}dx>0.$$ But this contradicts the fact that
    $$\lim_{M\to \infty}\int_{\{x:g(x)>0\}}|\sum_{n=-M}^{M}c_{n}e^{2\pi i nx}|^{2}dx=0.$$ This proves the claim.

    Now we have that $\sum_{n=-M}^{M}c_{n}e^{2\pi i nx}$ converges to zero uniformly on $B$ since we can assume $f=0$ on $(c,d)$.
    But by assumption we have
    $$\int_{B}|\sum_{n=-M}^{M}c_{n}e^{2\pi i nx}-\chi_{B}|^{2}d\mu_{s}+\int_{B^{c}}|\sum_{n=-M}^{M}c_{n}e^{2\pi i nx}-\chi_{B}|^{2}d\mu_{s}\to 0$$ as $M\to \infty$ so that by uniform convergence on $B$, we must have $\mu_{s}(B)=0$, which is a contradiction.
\end{proof}

\section{A Fourier dextrodual example for the real line}
We now prove the existence of a frame-like Fourier expansion for singular Borel probability measures on $\mathbb{R}$ that is of different nature than our dextrodual expansions from the last section. This expansion arises from Kaczmarz algorithms for Hilbert spaces as well as the Rokhlin disintegration theorem to relate measures on $\mathbb{R}$ to measures on $[0,1)\times \mathbb{Z}$.
\subsection{The case for the torus}
To show this expansion, we first recall the techniques of the case of the torus.
In \cite{Herr2017Fourier}, Herr and Weber prove Theorem \ref{singular} using the Kaczmarz algorithm for Hilbert spaces:
\begin{defn}
Given a complete sequence $\{e_{n}\}_{n=0}^{\infty}$ in Hilbert space $H$, define the 
\textbf{auxiliary sequence} associated with $\{e_{n}\}$ recursively:
$$g_{0}=e_{0}$$ and for $n\geq 1$
$$g_{n}=e_{n}-\sum_{k=0}^{n-1}\langle e_{n},e_{k}\rangle g_{k}.$$

The sequence $\{e_{n}\}$ is called \textbf{effective} if for all $f\in H$,
$$\sum_{k=0}^{n}\langle f,g_{k}\rangle e_{n}\to f$$ in $H$.
\end{defn}
Furthermore, effectiveness of the sequence $\{e_{n}\}_{n=0}^{\infty}$ in a Hilbert space $H$ is equivalent to $\{g_{n}\}_{n=0}^{\infty}$ being a Parseval frame by a result from Haller and Szwarc \cite{Haller2005Kaczmarz}.

The authors in \cite{Herr2017Fourier} show the effectiveness of $\{e^{2\pi i nx}\}_{n=0}^{\infty}$ using the following Theorem:

\begin{thm}[Kwapień and Mycielski]\label{KM}
A stationary sequence of complete unit vectors in a Hilbert space is effective if and only if its spectral measure is singular with respect to Lebesgue measure or normalized Lebesgue measure.
\end{thm}

Recall that given a stationary sequence of unit vectors in a Hilbert space $H$, $\{e_{n}\}$, the \textbf{spectral measure} $\mu$ is the probability measure on $[0,1)$ such that
$$\langle e_{0},e_{k}\rangle =\int_{0}^{1}e^{-2\pi i kx}d\mu(x)$$
for all $k$ from the Bochner theorem.

\subsection{Two dimensional case}

Surprisingly, these Fourier expansions on the real line are inspired from the Fourier expansions in $[0,1)^{2}$ by Herr, Jorgenesen, and Weber in \cite{Herr2022Fourier}.
In this two dimensional paper, the authors use the Rokhlin disintegration \cite{Rohlin1949Fundamental,Rohlin1949Decomposition} to develop a notation of slice singular, which is stronger than singular.

\begin{thm}[Rokhlin disintegration]\label{rok}
For a Borel probability measure $\mu$ on $A\times B$ where $A$ and $B$ are metric spaces, there is a Borel probability measure $\mu_{b}$ on $B$ and a family of Borel probability measures $\gamma^{b}$ on $A$, indexed by $B$ such that
\begin{enumerate}
\item For $f(a,b)\in L^{1}(\mu)$ and $\mu_{b}$ almost every $b$, $f(a,b)\in L^{1}(\gamma^{b})$ and $\int_{A}f(a,b)d\gamma^{b}\in L^{1}(\mu_{b})$

\item For each $f\in L^{1}(\mu)$
$$\int_{A\times B}f(a,b)d\mu=\int_{B}\int_{A}f(a,b)d\gamma^{b}d\mu_{b}.$$
\end{enumerate}
\end{thm}

By the Rokhlin disintegration, for a Borel probability measure $\mu$ on $[0,1)\times S$ where $S$ is a Borel set in Euclidean space, $\mu$ has a family of Borel probability measures on $[0,1)$, $\gamma^{s}$ indexed by $S$ and a Borel probability measure $\mu_{s}$ on $S$ called the \textbf{marginal measure} such that the above properties hold. The authors in \cite{Herr2022Fourier} define a measure $\mu$ on $[0,1)^{2}$ to be slice singular if there is a disintegration where $\mu_{s}$ is singular and $\gamma^{s}$ is singular for $\mu_{s}$ almost every $s$. Part of the discussion of this section however is that $\mu_{s}$ being singular or being defined on $[0,1)$ is not always important for general Fourier expansions.

\begin{defn}
We will call Borel probability measure $\mu$ on $[0,1)\times S$ \textbf{$S$-slice singular} if for $\mu_{s}$ almost every $s$, $\gamma^{s}$ is singular where $\gamma^{s}$ and $\mu_{s}$ are defined from Theorem \ref{rok}.
\end{defn}

The following result they proved using a operator version of the Kaczmarz algorithm, and it is important for our discussion that this result holds regardless of any properties of $\mu_{s}$.

\begin{thm}[Herr, Jorgensen, and Weber]\label{2dlemma}
If $\mu$ is an $S$-slice singular measure on $[0,1)\times S$, then for all $f\in L^{2}(\mu)$,
$$f(x,s)=\sum_{n=0}^{\infty}\langle f(x,s),g_{n}^{s}(x)\rangle_{L^{2}(\gamma^{s})}e^{2\pi i nx}$$
where for each $s$, $g_{n}^{s}(x)$ is the auxiliary sequence associated with $L^{2}(\gamma^{s})$ and $\{e^{2\pi inx}\}$, and for each $n$, $$\langle f(x,s),g_{n}^{s}(x)\rangle_{L^{2}(\gamma^{s})}\in L^{2}(\mu_{s}).$$
\end{thm}

\subsection{Main result}

Now for our discussion for singular measures on the real line, we draw on a result from Forelli \cite{Forelli1975F} that gives us a completeness result of a lattice of frequencies of exponential functions. This result is an extension of the F. and M. Riesz theorem for the real line.

\begin{thm}[Forelli]\label{Forelli}
Let $a,b>0$ be linearly independent over $\mathbb{Z}$.
If $\mu$ is a finite Borel measure on $\mathbb{R}$ with
$$\hat{\mu}(an+bm)=\int_{\mathbb{R}}e^{-(an+bm)x}d\mu=0$$ for all $n,m\in \mathbb{Z}^{+}$, then $\mu$ is absolutely continuous with respect to Lebesgue measure.
\end{thm}

\begin{cor}\label{2complete}
If $\mu$ is a singular finite Borel measure on $\mathbb{R}$, then for $a,b>0$ that are linearly independent over $\mathbb{Z}$
$$\overline{span}\{e^{2\pi i(a^{-1}nx+b^{-1}mx)}\}: n,m\in \mathbb{N}\}=L^{2}(\mathbb{R},\mu)$$
\end{cor}
\begin{proof}
    Suppose that for some $f\in L^{2}(\mu)$
    $$\langle f,e^{2\pi i (a^{-1}nx+b^{-1}mx)}\rangle= \int_{\mathbb{R}}e^{-2\pi i(a^{-1}nx+b^{-1}mx)}fd\mu=0$$
     for all $n,m\in \mathbb{N}$,
    then by Theorem \ref{Forelli}, $fd\mu$ is an absolutely continuous measure so that by the Lebesgue Decomposition Theorem, $fd\mu$ is the zero measure, which implies $f=0$ $\mu$ almost everywhere.
\end{proof}
Now in order prove our Fourier expansion result for the real line, we use the effective sequence result of Kwapien and Mycielski.
\begin{thm}\label{Reffective}
For a Borel probability measure $\mu$ on $\mathbb{R}$ and $a> 0$,
$\{e^{2\pi i a^{-1}nx}\}_{n=0}^{\infty}$ is effective in
$\overline{span}\{e^{2\pi i a^{-1} nx}: n\in \mathbb{N}\}$ if and only if the measure $v$ defined on Borel sets of $[0,a)$ as follows:
$$v(E)=\sum_{n=-\infty}^{\infty}\mu(E+an)$$
is singular or normalized Lebesgue measure. In particular if $\mu$ is singular,
$v$ is singular.
\end{thm}
\begin{proof}
Note $\{{e^{2\pi i a^{-1}kx}}\}_{k=0}^{\infty}$ is a stationary sequence of unit vectors. Furthermore for each $k$, 
$$\langle 1,e^{2\pi i a^{-1}kx}\rangle_{\mu} =\int_{0}^{a}e^{-2\pi i a^{-1}kx}dv.$$
The first result follows from the Theorem \ref{KM}.

Now if $\mu$ is singular there is Borel set $A\subseteq \mathbb{R}$ such that
$\mu(A)=0=|A^{c}|.$
See
$\bigcap_{k=-\infty}^{\infty}[A\cap [ak,a(k+1))-ak]$ is $v$ measure zero, and for each $k$,
$$|[A\cap [ak,a(k+1))]-ak)|=|A\cap [ak,a(k+1))|=a-|A^{c}\cap [ak,a(k+1)))|=a.$$
Then $\bigcap_{k=-\infty}^{\infty}[A\cap [ak,a(k+1))-ak]$ has Lebesgue measure $a$, and $v$ is singular.

\end{proof}
We do have a completeness issue for the case of the real line from needing a lattice of frequencies, but to fix this, we consider a special case where we do have completeness of one frequency of exponential functions.
\begin{cor}\label{Z}
Suppose that $\mu$ is a Borel probability measure supported on $a+b\mathbb{Z}$ for some $a,b\in \mathbb{R}$. Let $c>0$ and not a rational multiple of $b$. 
Then for every $f\in L^{2}(\mu)$,
$$f=\sum_{n=0}^{\infty}\langle f, g_{n}\rangle e^{2\pi i c^{-1}nx}$$
where $\{g_{n}\}$ is the auxiliary sequence associated with $\{e^{2\pi i c^{-1}nx}\}$ and $L^{2}(\mu)$.
\end{cor}
\begin{proof}
Notice $$e^{2\pi i b^{-1}x}=e^{2\pi i b^{-1}a}$$ $\mu$ almost everywhere.
Then by Corollary \ref{2complete},
$$L^{2}(\mu)=\overline{span}\{e^{2\pi i c^{-1}nx}: n\in \mathbb{N}\}$$
where $c$ not a rational multiple of $b$.
The result follows from Theorem \ref{Reffective}.
\end{proof}
With this frame-like expansion for these discrete measures, a natural question one asks is if these expansions come from weighted Fourier frames since it is clear that these discrete measures can't have true Fourier frames. This next proposition shows this is never the case.
\begin{prop}
Let $\mu$ be a finite Borel probability measure on $2\pi \mathbb{Z}$ whose support is not bounded. For any $\{c_{n}\}\in \ell^{2}(\mathbb{N})$ with all positive entries,
$\{c_{n}e^{2\pi i nx}\}_{n=0}^{\infty}$ is a complete Bessel sequence in $L^{2}(\mu)$ and never a frame.
\end{prop}

\begin{proof}
Suppose $f\in L^{2}(\mu)$ and let the weights for the measure $\mu$ be $\{a_{n}\}_{n\in \mathbb{N}}$ under the usual ordering of $\mathbb{Z}$. We have

$$\sum_{k}|\sum_{n}a_{n}f(n)c_{k}e^{-2\pi i kn}|^{2}\leq \sum_{k}c_{k}^{2}[\sum_{n}a_{n}|f(n)|]^{2}=||\{c_{k}\}||_{\ell^{2}(\mathbb{N})}^{2}||f||_{L^{1}(\mu)}^{2}$$
$$\leq ||\{c_{k}\}||_{\ell^{2}(\mathbb{N})}^{2}||f||_{L^{2}(\mu)}^{2}.$$ This shows that $\{c_{n}e^{2\pi i nx}\}_{n=0}^{\infty}$ is a complete Bessel sequence in $L^{2}(\mu)$ since we already know the sequence is complete.

Now to show $\{c_{n}e^{2\pi i nx}\}_{n=0}^{\infty}$ does not have a lower frame bound,
consider the following sequence of norm one vectors:
$\{f_{n}\}=\{\frac{1}{\sqrt{a_{n}}}\chi_{\{n\}}\}$. We assume $a_{n}\neq 0$ for all $n$, otherwise we take a subsequence. We have
$$\sum_{k=0}^{\infty}|\langle f_{n},c_{k}e^{2\pi i kx}\rangle|^{2}=a_{n}||\{c_{k}\}||_{\ell^{2}(\mathbb{N})}^{2}\to 0$$ as $n\to \infty$.

\end{proof}

Now for our expansion result, we use the two dimensional result from Theorem \ref{2dlemma} and expand the coefficients using Corollary \ref{Z} and Theorem \ref{rok}.
\begin{cor}\label{intervalandz}
Let $\mu$ be a Borel probability measure on $[0,1)\times \mathbb{Z}$ that is $\mathbb{Z}$-slice singular. Then for all $f\in L^{2}(\mu)$,
$$f(x,k)=\sum_{n=0}^{\infty}\sum_{m=0}^{\infty}\langle f, g_{m}(k)g_{n}^{k}(x)\rangle_{\mu} e^{imk}e^{2\pi i nx}$$
where $g_{m}(k)$ is the auxiliary sequence associated with the marginal measure on $\mathbb{Z}$ and $\{e^{imk}\}$.

\end{cor}
This result may seem random at first, but we will now show that singular measures on $\mathbb{R}$ are morally the same as certain $\mathbb{Z}$ slice singular measures. What follows is the statement of existence of frame-like expansions for singular Borel measures on the real line.
\begin{thm}
For any singular Borel probability measure $\mu$ on $\mathbb{R}$, there is a Parseval frame $\{c_{mn}\}_{n,m=0}^{\infty}\subseteq L^{2}(\mu)$ such that for any $f\in L^{2}(\mu)$,

$$f(y)=\sum_{n=0}^{\infty}\sum_{m=0}^{\infty}\langle f, c_{nm}\rangle_{\mu}e^{im \lfloor y \rfloor}e^{2\pi i ny}$$ where the right limit in taken in $L^{2}(\mu)$ first for each $n$, then the left limit is taken in $L^{2}(\mu)$.
\end{thm}
Notice that in this Theorem, $\{c_{mn}\}_{n,m=0}^{\infty}$ is not a exactly dextrodual to $\{e^{im \lfloor y \rfloor}e^{2\pi i ny}\}$, but this expansion is Fourier frame-like as it shows $f$ has a Fourier dextrodual expansion on each $[k,k+1)$. More precisely, there is sequence $\{d_{kn}\}\subseteq L^{2}(\mu)$ where $$\sum_{k\in \mathbb{Z}}\int_{[k,k+1)}|f(y)-\sum_{n=0}^{M}\langle f,d_{kn}\rangle_{\mu}e^{2\pi i ny}|^{2}d\mu \to 0$$ as $M\to \infty.$
\begin{proof}
Define a $\mathbb{Z}$ slice singular Borel probability measure $\sigma$ on $[0,1)\times \mathbb{Z}$ as follows:
$$\sigma(B)=\sum_{n\in \mathbb{Z}}\mu(B_{n}+n)$$
where $B_{n}$ is the $n$-slice of $B$.
To see that $\sigma$ is $\mathbb{Z}$ slice singular, one can see that with the Rokhlin distintegration, the marginal measure $\mu_{n}$ on $\mathbb{Z}$ is $\sum_{n}\mu([n,n+1))\delta_{n}$ and the slice measures on $[0,1)$ indexed by $\mathbb{Z}$ are defined by $\gamma^{n}(B)=\frac{\mu(B+n)}{\mu([n,n+1))},$ which is defined for $\mu_{n}$ almost every $n$ and singular for $\mu_{n}$ almost every $n$.

Now let $T: L^{2}(\mu,\mathbb{R})\to L^{2}(\sigma,[0,1)\times \mathbb{Z})$ where
$$T(f(y))= F(x,k)=f(x +k).$$
This map is an isometry because for all $f\in L^{2}(\mu)$,
$$\int_{\mathbb{Z}}\int_{[0,1)}|f(x+k)|^{2}d\gamma^{k}d\mu_{k}=\sum_{k: \mu([k,k+1))>0}\int_{[0,1)}|f(x+k)|^{2}d[\mu+k]=||f||_{L^{2}(\mu)}^{2},$$
and for any finite sequence $\{a_{nm}\}$,
$$T(\sum_{n}\sum_{m}a_{nm}e^{im\lfloor y\rfloor}e^{2\pi i ky})= \sum_{n}\sum_{m}a_{nm}e^{imk}e^{2\pi i kx}.$$
Then for $f(y)\in L^{2}(\mu)$ as a result of Corollary \ref{intervalandz},
$$f(y)=\sum_{n=0}^{\infty}\sum_{m=0}^{\infty}\langle f,T^{*}(g_{m}(k)g_{n}^{k}(x))\rangle_{\mu}e^{im \lfloor y \rfloor}e^{2\pi i ny}.$$

Now to show $\{T^{*}(g_{m}(k)g_{n}^{k}(x))\}$ is a Parseval frame,
consider the following for any $f\in L^{2}(\mu)$:
\begin{equation}
\begin{split}
\sum_{n}\sum_{m}|\langle T(f),g_{m}(k)g_{n}^{k}(x)\rangle_{\sigma}|^{2} & =\sum_{n}\sum_{m}|\langle \langle T(f),g_{n}^{k}(x)\rangle_{\gamma^{k}},g_{m}(k)\rangle_{\mu_{k}}|^{2} \\
& =
\sum_{n}||\langle T(f),g_{n}^{k}\rangle_{\gamma^{k}}||_{\mu_{k}}^{2} \\
& =\int_{\mathbb{Z}}\sum_{n}|\langle T(f),g_{n}^{k}\rangle_{\gamma^{k}}|^{2}d\mu_{k}\\ 
& =\int_{\mathbb{Z}}||T(f)||_{\gamma^{k}}^{2}d\mu_{k}\\
&=
||f||_{\mu}^{2}. 
\end{split}
\end{equation}
This computation follows by the Rokhlin disintegration theorem and since $\{g_{m}\}$ is a Parseval frame in $L^{2}(\mu_{k})$ and $\{g_{n}^{k}\}$ is a Parseval frame in $L^{2}(\gamma^{k})$ for $\mu_{k}$ almost every $k$.
\end{proof}

\section{Subspaces of $h^{2}(\mathbb{D})$ generated by Bessel sequences.}
We continue our discussion by considering Bessel sequences and their connection to subspaces of harmonic functions on the unit disk, including Bessel sequences in $L^{2}(\mu)$ that are dextrodual to $\{e^{2\pi inx}\}$.
In particular, we close this section by showing that if $\mu$ has the property from equation (\ref{prp}), then all $f\in L^{2}(\mu)$ are $L^{2}(\mu)$ boundaries of harmonic functions whose coefficients come from the dextrodual sequence.

\begin{defn}
Recall one can define $h^{2}(\mathbb{D})$ as the set of complex valued harmonic functions on the disk whose power series coefficients are square summable.
Also, recall that $H^{2}(\mathbb{D})$, the set of functions in $h^{2}(\mathbb{D})$ that are analytic, is a closed subspace of $h^{2}(\mathbb{D})$.

Let $\{g_{n}\}_{n=-\infty}^{\infty}$ be a Bessel sequence in Hilbert space $H$.
Define $A_{g_{n}}: H\to h^{2}(\mathbb{D})$ where
$$A_{g_{n}}(f)=\sum_{n=-\infty}^{\infty}\langle f,g_{n}\rangle r^{|n|}e^{2\pi i n\theta}$$ and $H(g_{n})=Im(A_{g_{n}}).$
\end{defn}

\subsection{Correspondence with model spaces}
We will discuss specifically connections between $H(g_{n})$ and models spaces, and we will write $\{g_{n}\}_{n=0}^{\infty}$ to denote $\{g_{n}\}_{n=0}^{\infty}\cup \{0\}_{n=-\infty}^{-1}$. We review the normalized Cauchy transform here:

\begin{defn}
The \textbf{Cauchy transform} of a finite Borel measure $\mu$ on $[0,1)$, we denote $C_{\mu}$ is a map from $L^{2}(\mu)$ to analytic functions on the unit disk given by
$$C_{\mu}(f)=\int_{0}^{1}\frac{fd\mu}{1-ze^{-2\pi i x}}=\sum_{n=0}^{\infty}\langle f, e^{2\pi inx}\rangle z^{n}$$ and the \textbf{normalized Cauchy transform} for $\mu$, denoted $V_{\mu}$ is given by
$$V_{\mu}(f)=\frac{C_{\mu}(f)}{C_{\mu}(1)}=\sum_{n=0}^{\infty}\langle f, h_{n}\rangle z^{n}$$ where $h_{n}\in L^{2}(\mu)$ are in $span\{e^{2\pi i kx}\}$ for each $n$.
\end{defn}
By work of Sarason \cite{Sarason1994Sub-Hardy}, when $\mu$ is singular $V_{\mu}$ is an isometry onto $$H(b)=H^{2}(\mathbb{D})\ominus b(z)H^{2}(\mathbb{D})$$ where $b(z)$ is the inner function corresponding to $\mu$ via the Herglotz theorem.
Furthermore, the model spaces $H(b)$ are exactly the proper backward-shift-invariant subspaces of $H^{2}(\mathbb{D})$ by work of Beurling \cite{Beurling1948Two}.
Our main result for when $H(g_{n})$ is a model space follows, which is also discussed by Christensen in \cite{Christensen2020Frame}.
\begin{prop}\label{classify}
If $\{g_{n}\}_{n=0}^{\infty}$ is a frame and not a Riesz basis, then $H(g_{n})=H(b)$ for some inner function
$b(z)$ if and only if there exists a bounded operator that maps $g_{n}$ to $g_{n+1}$ for all $n$.
\end{prop}

\begin{proof}
Suppose that $H(g_{n})=H(b)$ for some inner function
$b(z)$. Then $H(g_{n})$ is backward-shift-invariant, meaning for all $f\in H$, there is a $U(f)\in H$ such that
$$\sum_{n=0}^{\infty}\langle f,g_{n+1}\rangle (re^{2\pi i \theta})^{n}=\sum_{n=0}^{\infty}\langle U(f),g_{n}\rangle (re^{2\pi i \theta})^{n}.$$
By linearity of the inner product and since $\{g_{n}\}$ is complete, $U$ is linear. $U$ is also bounded because $\{g_{n}\}$ is a frame. Therefore, $U^{*}(g_{n})=g_{n+1}$ for all $n$.

Now suppose that there exists a bounded operator $T$ such that $T(g_{n})=g_{n+1}$ for all $n$.
Then take any $f\in L^{2}(\mu)$. We have
$$\sum_{n=0}^{\infty}\langle f, g_{n+1}\rangle (re^{2\pi i \theta})^{n}=\sum_{n=0}^{\infty}\langle T^{*}(f), g_{n}\rangle (re^{2\pi i \theta})^{n}.$$ Then $H(g_{n})$ is backward shift invariant.

\end{proof}

These spaces also have a nice reproducing kernel:

\begin{prop}
Let $\{g_{n}\}_{n=0}^{\infty}$ be a frame in Hilbert space $H$ and $S$ be the frame operator of $\{g_{n}\}$. Then the reproducing kernel of $H(g_{n})$ for each $w\in \mathbb{D}$ is 
$$K_{w}(z)=\sum_{k=0}^{\infty}\sum_{n=0}^{\infty}\langle S^{-1}(g_{n}),g_{k}\rangle (\overline{w})^{n}z^{k},$$
meaning for all $w\in \mathbb{D}$ and $f(z)\in H(g_{n})$
$$f(w)=\langle f(z),K_{w}(z)\rangle_{H(g_{n})}$$ where $z=re^{2\pi i \theta}$.
\end{prop}
\begin{proof}
Notice that $A^{*}_{g_{n}}A_{g_{n}}=S$ and $A_{g_{n}}S^{-1}A^{*}_{g_{n}}$ is the orthogonal projection onto $H(g_{n})$.
Therefore, since $\sum_{n=0}^{\infty}(\overline{w})^{n}z^{n}$ is the reproducing kernel of $H^{2}(\mathbb{D})$,  $$A_{g_{n}}S^{-1}A^{*}_{g_{n}}\sum_{n=0}^{\infty}(\overline{w})^{n}z^{n}=\sum_{k=0}^{\infty}\sum_{n=0}^{\infty}\langle S^{-1}(g_{n}),g_{k}\rangle (\overline{w})^{n}z^{k}$$ is the reproducing kernel of $H(g_{n})$.
\end{proof}
The next result is a classification of frames using the normalized Cauchy transform.
\begin{cor}
    Suppose that $\{g_{n}\}_{n=0}^{\infty}$ is a frame and not a Riesz basis for Hilbert space $H$ and there exists $B\in B(H)$ such that $B(g_{n})=g_{n+1}$ for all $n$. Then $\{g_{n}\}$ is equivalent to a sequence $\{h_{n}\}$ defined by $V_{\mu}$ where $\mu$ is singular.
\end{cor}
\begin{proof}
By Proposition \ref{classify}, there is inner function $b(z)$ so that $H(g_{n})=H(b).$ Now let $\{h_{n}\}$ be the Parseval frame defined from $V_{\mu}$ where $\mu$ is the finite singular measure on $[0,1)$ associated with $b(z)$ via the Herglotz theorem. By uniqueness of reproducing kernels, we have for every $n$ and $k$,
$$\langle S^{-1}g_{n},g_{k}\rangle_{H}=\langle h_{n},h_{k}\rangle_{L^{2}(\mu)}$$ since the frame operator for a Parseval frame is the identity.

It follows that since $S$ is positive definite, $\{S^{-\frac{1}{2}}(g_{n})\}$ is unitarly equivalent to $\{h_{n}\}$.
\end{proof}

\subsubsection{Examples}

We close this subsection with a couple of examples of Parseval frames and their connection to model spaces.

\begin{ex}
Let $H$ be a Hilbert space with orthonormal basis $\{e_{n}\}_{n=0}^{\infty}$.
We have the set
$$\{g_{n}\}_{n=0}^{\infty}=\{e_{0},e_{1},\frac{1}{\sqrt{2}}e_{2},\frac{1}{\sqrt{2}}e_{2}, \frac{1}{\sqrt{3}}e_{3},\frac{1}{\sqrt{3}}e_{3},\frac{1}{\sqrt{3}}e_{3},\dots\}$$
is a Parseval frame for $H$ and $H(g_{n})$ is not a model space by Proposition \ref{classify}.
\end{ex}

\begin{ex}
Let $\mu$ be a singular Borel probability measure on $\mathbb{R}$ and let $\{g_{n}\}$ be the auxiliary sequence associated with Hilbert space
$$\overline{span}\{e^{2\pi i nx}: n\in \mathbb{N}\}\subseteq L^{2}(\mu)$$ and sequence $\{e^{2\pi i nx}\}_{n=0}^{\infty}$, we have $$\overline{span}\{e^{2\pi i nx}: n\in \mathbb{N}\}(g_{n})=H(b)$$ where $b(z)$ is the inner function corresponding to the measure $v$ via the Herglotz theorem where
$$v(E)=\sum_{k=-\infty}^{\infty}\mu(E+k).$$  This follows from the fact that sequences $\{e^{2\pi i nx}\}_{n=0}^{\infty}$ are unitarly equivalent in $L^{2}(\mu)$ and $L^{2}(v)$. The auxiliary sequences in each measure space are also unitarly equivalent as a result.
\end{ex}

\subsection{Correspondence with subspaces of $h^{2}(\mathbb{D})$}
Another result from Beurling \cite{Beurling1948Two} is that the subspaces of $h^{2}(\mathbb{D})$ that are invariant under multiplication by $z$ and $\overline{z}$ are exactly $f(z)h^{2}(\mathbb{D})$ where $|f|$ is zero or one almost everywhere on $\partial \mathbb{D}$.
We have a similar result from techniques from this section:
\begin{prop}
For a Hilbert space $H$ with frame $\{g_{n}\}_{n=-\infty}^{\infty}$, the following are equivalent:
\begin{enumerate}
\item $H(g_{n})=f(z)h^{2}(\mathbb{D})$ where $|f|$ is zero or one almost everywhere on $\partial \mathbb{D}$.
    \item $\{g_{n}\}=\{T^{n}g_{0}\}$ for some bounded invertible operator $T$.
    \item $\{g_{n}\}$ is similar to Parseval frame $\{\chi_{E}e^{2\pi i nx}\}_{n=-\infty}^{\infty}$ in Hilbert space $\chi_{E}L^{2}([0,1))$ for some $E\subseteq [0,1)$ measurable.
\end{enumerate}.
\end{prop}

\subsubsection{Examples}
Now we discuss examples of subspaces of $h^{2}(\mathbb{D})$ arising from $H(g_{n})$ where $\{g_{n}\}$ is dextrodual to $\{e^{2\pi inx}\}$. 

\begin{ex}
    If $\mu$ is an absolutely continuous Borel measure on $[0,1)$ with Radon-Nikodym  derivative $g$ bounded above and below on its support, it is clear that 
    $$L^{2}(\mu)(\chi_{\{x: g(x)>0\}}\frac{e^{2\pi i nx}}{g})=\chi_{\{x: g(x)>0\}}L^{2}([0,1))=f(z)h^{2}(\mathbb{D})$$ where $f(z)=\sum_{n}\langle \chi_{\{x: g(x)>0\}},e^{2\pi i nx}\rangle r^{|n|}e^{2\pi i nx}$.
\end{ex}
\begin{ex}
Let $\mu=\sum_{k=1}^{n}a_{k}\delta_{b_{k}}+\mu_{a}$ be a measure from Proposition \ref{discretecase}. It is clear that the $H(g_{n})$ from this proposition is not invariant under multiplication by $z$ or $\overline{z}$ since the $\tilde{f}$ in Proposition \ref{discretecase} is a step function on an interval.
\end{ex}
\subsubsection{Boundary behavior of $A_{g_{n}}$}
We conclude with our statement about our dextroduality relating to boundary behavior of $A_{g_{n}}$.

\begin{thm}
Let $\mu$ be a finite Borel measure on $[0,1)$ such that there is Bessel sequence $\{g_{n}\}_{n=-\infty}^{\infty}$ that is dextrodual to $\{e^{2\pi inx}\}$. 
Then for all $f\in L^{2}(\mu)$,
$$A_{g_{n}}(f)(re^{2\pi i \theta})\to f$$ in $L^{2}(\mu)$ as $r\to 1^{-}.$
\end{thm}
In other words, if $\mu$ has the property from equation (\ref{prp}), then for each $f\in L^{2}(\mu)$, $A(g_{n})(f)$ is continuous in $r$ on $[0,1]$ into $L^{2}(\mu)$.
\begin{proof}
We need to show for all $f\in L^{2}(\mu)$,
$$\sum_{n}\langle f,g_{n}\rangle r^{|n|}e^{2\pi i n \theta}\to \sum_{n}\langle f,g_{n}\rangle e^{2\pi i n \theta}$$ in $L^{2}(\mu)$ as $r\to 1^{-}.$
This follows from Abel summability in Hilbert spaces since $\sum_{n}\langle f,g_{n}\rangle r^{|n|}e^{2\pi i n \theta}$ exists in $L^{2}(\mu)$ for all $r\in [0,1]$ by the dextroduality property and fact that 
$\sum_{n}|\langle f,g_{n}\rangle|^{2}<\infty$ for all $f\in L^{2}(\mu)$.
\end{proof}

\printbibliography
\end{document}